\documentclass[10pt,a4paper,reqno]{article}
\usepackage[T1]{fontenc}
\usepackage[all]{xy}
\usepackage{epsfig}
\usepackage{amssymb}
\usepackage{amsmath}
\usepackage{graphics,graphicx}
\usepackage{eepic}
\usepackage{epsfig}

\usepackage{vmargin}
\usepackage{theorem}
\usepackage{amsbsy}
\usepackage{amsfonts}
\usepackage{fig4tex}

\newtheorem{theo}{Theorem}[section]
\newtheorem{theorem}[theo]{Theorem}
\newtheorem{lemma}[theo]{Lemma}
\newtheorem{prop}[theo]{Proposition}

{\theorembodyfont{\rmfamily}

\newtheorem{remark}[theo]{Remark}}

\newenvironment{proof}{\noindent\text{\textbf{Proof.\:}}}{}

\def\qed{\hfill $\square$ \goodbreak}

\newcommand{\N}{\mathbb{N}}
\newcommand{\R}{\mathbb{R}}
\newcommand{\CC}{\mathbb{C}}
\newcommand{\C}{\mathcal{C}}
\renewcommand{\H}{\mathbb{H}}

\newcommand{\Om}{\Omega}

\newcommand{\G}{\Gamma}
\newcommand{\la}{\lambda}
\newcommand{\Arg}{{\rm Arg}}

\def\eps{\varepsilon}

\title{About H\"older-regularity of the convex shape minimizing $\lambda_2$}
\author{Jimmy Lamboley\footnote{Ceremade (UMR CNRS 7534), Universit\'e Paris-Dauphine, Place du Mar\'echal de Lattre de Tassigny, 75775 Paris C\'edex 16, France. E-mail: {\tt lamboley@ceremade.dauphine.fr}}}

\begin{document}
\maketitle
\begin{abstract}
In this paper, we consider the well-known following shape optimization problem:
$$\lambda_2(\Omega^*)=\min_{\stackrel{|\Omega|=V_0} {\Omega\textrm{ convex}}} \lambda_2(\Omega),$$
 where $\lambda_2(\Om)$ denotes the second eigenvalue of the Laplace operator with homogeneous Dirichlet boundary conditions in $\Om\subset\R^2$, and $|\Om|$ is the area of $\Om$. We prove, under some technical assumptions, that any optimal shape $\Omega^*$ is $\mathcal{C}^{1,\frac{1}{2}}$ and is not $\C^{1,\alpha}$ for any $\alpha>\frac{1}{2}$. We also derive from our strategy some more general
regularity results, in the framework of partially overdetermined boundary value problems, and we apply these results to some other shape optimization problems.
\\

\smallskip

\end{abstract}

\section{Main result}\label{sect:main}

In this paper, we prove an optimal regularity result for the shape which minimizes the second eigenvalue of the 2-dimensional Laplacian, with homogeneous Dirichlet boundary conditions, under volume and {\bf convexity} constraints. Moreover, we make good use of the tools introduced to that end and we give some more general results about regularity of overdetermined elliptic PDE. Finally we apply these ones to some other shape optimization problems.

Let us first introduce our notations. All the results of this paper involve subsets of
$\R^2$, and $|\cdot|$ denotes the Lebesgue measure in $\R^2$.
Let $\Omega$ be an open set, with finite area in the plane, and let us denote by
$$0<\lambda_1(\Omega)\leq\lambda_2(\Omega)\leq\lambda_3(\Omega)\leq\ldots$$
its eigenvalues for the Laplace operator with homogeneous Dirichlet boundary conditions (Dirichlet-Laplacian).\\

Here, we are mainly interested in studying the regularity of the solution of the following shape optimization problem :
\begin{equation}\label{eq:pb}
\Omega^*\textrm{ an open convex set, such that }|\Omega^*|=V_0,
\textrm{ and }\lambda_2(\Omega^*)=\min_{\stackrel{|\Omega|=V_0} {\Omega\textrm{ convex}}} \lambda_2(\Omega),
\end{equation}
where $V_0$ is a given positive real number.

A theorem by Krahn and Szeg\"o asserts that the solution of problem (\ref{eq:pb}) with no convexity constraint
is the disjoint union of two identical balls (this is an easy consequence of the so-called Faber-Krahn Theorem which asserts that the shape minimizing the first eigenvalue among sets of prescribed volume is a ball, see Figure \ref{Minlambda} below).
The problem (\ref{eq:pb}) with the convexity constraint is studied in \cite{HO}: they prove
the existence and some geometric properties of optimal shapes $\Omega^*$.
 In particular, they show  that the stadium (i.e. the convex hull of two identical tangent disks of suitable area) is not a solution, whereas it was expected and supported by numerical experiments (see e.g. \cite{T}). They also
prove, under some assumptions about the regularity and the geometry of $\Omega^*$, some optimality conditions satisfied by $\Omega^*$ (see Section \ref{sect:HO}; see also \cite{O} for numerical results, showing that the optimal shape for problem \eqref{eq:pb} is different, but close to the stadium).\\

\begin{figure}[!ht]
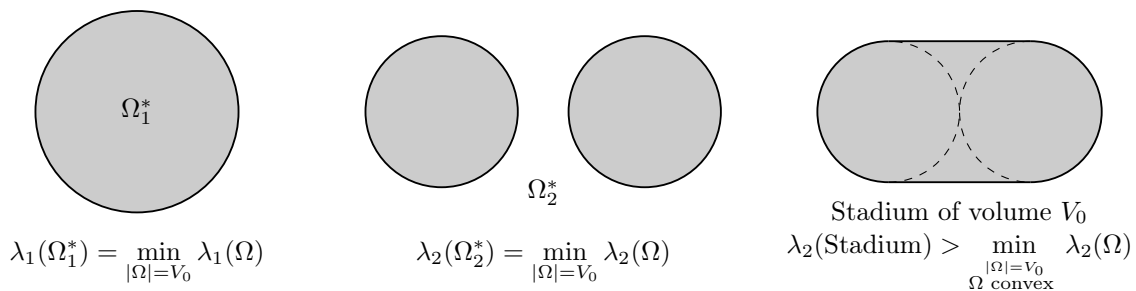

\figinit{0.38pt}
\figpt 0:(0,0)

\figpt 1:(300,0)
\figpt 2:(500,0)

\figpt 3:(740,0)
\figpt 4:(880,0)
\figpt 5:(740,70)
\figpt 6:(880,70)
\figpt 7:(740,-70)
\figpt 8:(880,-70)

\figpt 9:(400,-80)
\figpt 10:(810,-100)

\figpt 11:(0,-150)
\figpt 12:(400,-150)
\figpt 13:(810,-150)

\psbeginfig{}

\psset(fillmode=yes,color=0.8)
\pscirc 0(100)
\pscirc 1(75)
\pscirc 2(75)

\psarccirc 3 ; 70 (90,270)
\psarccirc 4 ; 70 (-90,90)
\psline[5,6,8,7]

\psset(width=0.7)
\psset(fillmode=no,color=0)

\pscirc 0(100)
\pscirc 1(75)
\pscirc 2(75)

\psarccirc 3 ; 70 (90,270)
\psarccirc 4 ; 70 (-90,90)
\psline[5,6]
\psline[7,8]

\psset(dash=8)
\psset(width=0.4)

\psarccirc 3 ; 70 (-90,90)
\psarccirc 4 ; 70 (90,270)

\psendfig
\figvisu{\figBoxA}{}{
\figwritec [0]{$\Omega^*_1$}
\figwritec [9]{$\Omega^*_2$}
\figwritec [10]{Stadium of volume $V_0$}
\figwritec [11]{$\displaystyle{\lambda_1(\Omega^*_1)=\min_{|\Om|=V_0} \lambda_1(\Om)}$}
\figwritec [12]{$\displaystyle{\lambda_2(\Omega^*_2)=\min_{|\Om|=V_0} \lambda_2(\Om)}$}
\figwritec [13]{$\displaystyle{\lambda_2(\textrm{Stadium})>\min_{\stackrel{|\Om|=V_0}{\Om\textrm{ convex }}} \lambda_2(\Om)}$}
}
\centerline{\box\figBoxA}
\caption{Minimization of the first two eigenvalues under volume constraint}
\label{Minlambda}
\end{figure}

We address here the question of the \textbf{regularity} of an optimal shape $\Omega^*$ for problem \eqref{eq:pb}.

The main result of this paper is the following theorem, which gives a negative answer to the open problem 7. of A. Henrot \cite{H}.
\begin{theorem}\label{th:convex}
Let $V_0>0$ and let $\Omega^*\subset\R^2$ be a solution of the minimization problem $(\ref{eq:pb})$,
that is to say an optimal convex set of given area for the second Dirichlet-Laplacian eigenvalue.\\
We assume:
\begin{equation}\label{H2}
\Om^*\textrm{ contains at most a finite number of segments in its boundary.}
\end{equation}
Then
\begin{equation}\label{eq:result}
\Omega^*\textrm{ is }\mathcal{C}^{1,\frac{1}{2}}\textrm{, and }\;\forall\;\eps>0,\;\Omega^*\textrm{ is not }\mathcal{C}^{1,\frac{1}{2}+\eps}.
\end{equation}
\end{theorem}

\begin{remark}
So far the $\mathcal{C}^1$-regularity of $\Om^*$ was known (see  \cite{B}), which excludes polygons for example. Here, this regularity is improved, and a most surprising part is that $\Om^*$ cannot be more than $\mathcal{C}^{1,\frac{1}{2}}$. More precisely, this ``singularity'' appears exactly at the junction between flat parts and
strictly convex parts of the boundary.
\end{remark}
\begin{remark} {\bf About assumption \eqref{H2}:}
 the boundary of a convex shape contains two specific subsets: on one hand the union of flat parts, and on the other hand the set
 \begin{equation}\label{eq:strict}
 \G:=\{x\in\partial\Om^*\;/\;\exists r>0 \textrm{ such that }B_{r}(x)\cap\Om^*\textrm{ is strictly convex}\}
 \end{equation}
  which is a relatively open subset of $\partial\Om^*$, and which will improperly be called the strictly convex parts of the boundary. We know that the flat parts of $\partial\Om^*$ are not empty, since using an argument on the nodal line of the second eigenfunction in a convex set, it is proven in \cite{HO} that there are at least two segments in the boundary.
  On the other hand, concerning the strictly convex parts of $\partial\Om^*$, it is not clear without assumption \eqref{H2} that this part is nonempty (even if we know that $\Om^*$ is not a polygon: see assumption \eqref{eq:H2} in Proposition \ref{prop:HO} and Remark \ref{rem:segments} where we exhibit a convex $\C^1$-set whose strictly convex parts are empty).\\
Concerning the regularity, each of these specific parts of the boundary is very smooth if $\Om^*$ is optimal (see Proposition \ref{prop:regCinfini}), so the singularity stated in \eqref{eq:result} is localized at junction points between a segment and a strictly convex part. Our analysis is local at these junction points, and this explains the technical assumption \eqref{H2} we made.
Particularly, we also prove in this paper that:
\begin{prop}\label{prop:regCinfini}
Under the assumptions of Theorem \ref{th:convex}, $\partial\Om^*$ is $\mathcal{C}^\infty$, except on a finite number of points, where the regularity is exactly $\C^{1,\frac{1}{2}}$.
\end{prop}
This is a consequence of Proposition \ref{prop:reg-over} and Theorem \ref{th:convex}. The regularity of the strictly convex parts is not new, and one can even have piecewise analyticity of the boundary, see \cite{KN,V} and Remark \ref{rem:reg-over}. However we give in this paper a new proof of the $\C^\infty$-regularity of these strictly convex parts (Proposition \ref{prop:reg-over}) to show the efficacy of our strategy (in dimension 2 only).
We discuss again this assumption \eqref{H2} in Remark \ref{rem:segments} and Remark \ref{rem:nonregC2}.\\
\end{remark}

There are three main steps in the proof of Theorem \ref{th:convex}:
\begin{itemize}
\item the first one is classical and uses \cite{HO}: writing optimality condition for \eqref{eq:pb}, one prove that any second eigenfunction in $\Om^*$ (an optimal set for this problem), is solution of a so-called partially overdetermined problem:
\begin{equation}
\left\{\begin{array}{ccll}
-\Delta u_2 &=& \lambda_2(\Om^*) u_2 & in\;\Omega^*\\
u_2&=&0& on\;\partial\Omega^*\\
|\nabla u_2| &=& C^{st}=\Lambda>0& on\; \Gamma,
\end{array}
\right.
\end{equation}
where $\G$ denotes the strictly convex parts of the boundary (in the sense of \eqref{eq:strict});
\item the second step, which is the main contribution of this paper, is to analyze the regularity of the junction between $\Gamma$ and $\partial\Om^*\setminus\Gamma$; we show that this regularity is either $\mathcal{C}^{1,\frac{1}{2}}$ or $\mathcal{C}^{2,\frac{1}{2}}$,
\item the third step is to prove that $\Om^*$ cannot be more than $\mathcal{C}^2$, using an result mainly due to Henrot and Oudet in \cite{HO}, see Proposition \ref{prop:nonregC2}.
\end{itemize}

In the following section, we remind some results of A. Henrot and E. Oudet from \cite{HO}, which lead to the optimality condition for problem \eqref{eq:pb}, then we prove Theorem \ref{th:convex} in section \ref{sect:convex}. In  the last section we give some comments on the spirit of the proof, which goes beyond this specific optimization problem ; thus we state a few other regularity results, and we apply these ones to some other shape optimization problems.

\section{First order optimality condition}\label{sect:HO}

In order to prove Theorem \ref{th:convex}, we want to write optimality conditions for problem \eqref{eq:pb}. We adapt the proofs given in \cite{HO} to get:
\begin{prop}\label{prop:HO}[Henrot-Oudet]
Let $\Omega^*$ be a solution of problem $(\ref{eq:pb})$, and let $u_2$ be one second eigenfunction in $\Omega^*$. 
 We assume that:
 \begin{equation}\label{eq:H2}
\partial\Om^*\textrm{ contains at least one nonempty relatively open strictly convex part (in the sense of \eqref{eq:strict}).}
 \end{equation}
 Then,
\begin{itemize}
\item
$\lambda_2(\Omega^*)$ is simple,
\item
we have an optimality condition on the strictly convex parts $\G$ of $\partial\Om^*$:
\begin{equation}\label{eq:EL}
|\nabla u_2|_{|\G}=\Lambda:=\sqrt{\frac{\lambda_2(\Omega^*)}{|\Omega^*|}}>0.
\end{equation}
\end{itemize}
\end{prop}
\begin{proof}
We first apply Theorem 5 in \cite{HO}, which asserts that $\la_2(\Om^*)$ is simple when $\Om^*$ is an optimal shape for \eqref{eq:pb}. The authors make a regularity assumption on $\Om^*$, namely the $\mathcal{C}^{1,1}$-regularity of the boundary. However, this technical assumption
can easily be avoided in their proof of Lemma 1 in \cite{HO}, which is the main tool of the proof of their Theorem 5 we are interested in: to see this, the main remark is that, thanks to the convexity of $\Om^*$, we know that the second eigenfunctions in $\Om^*$ belongs to $H^2(\Om^*)$ (see \cite{G} for instance), and so their normal derivatives are well defined in $H^{\frac{1}{2}}(\partial\Om^*)$ in the sense of trace on $\partial\Om^*$; this allows the computations of the directional derivatives of $\la_2$ used in the proof of Lemma 1 in \cite{HO}. Nevertheless, this part of the proof uses the assumption \eqref{eq:H2}, even if this one is not specified in \cite{HO} (see Remark \ref{rem:segments} below): indeed, they need the existence of a strictly convex part to perturbe the optimal shape around this part, and then write optimality.

We now apply the first part of Theorem 7 in \cite{HO} which gives equation \eqref{eq:EL}; once again, this result does not need any regularity assumption on $\Om^*$, since the $H^2$-regularity of the second eigenfunction is enough to write the shape derivative of the shape functional $\la_2$.\qed
\end{proof}
\begin{remark}\label{rem:segments}
The hypothesis \eqref{eq:H2} is not specified in \cite{HO}, but this one is implicitly used in the proof of Lemma 1  in \cite{HO} (and so this hypothesis is also needed for their Theorem 5 which is a direct consequence of this lemma). We point out that this property \eqref{eq:H2} is not satisfied by a general convex set,
 even assumed to be $\mathcal{C}^{1,1}$ like in \cite{HO}. In order to convince the reader of the existence of such a ``singular'' set,
  let us take a one-dimensional function $f$ such that $f''=\chi_\omega$, where $\omega$ is a closed subset of $\R$ with positive measure, and with an empty interior. Then the graph of $f$ is convex and $\mathcal{C}^{1,1}$, but
  there is an infinite number of segments in the boundary,
 and these ones even form a dense subset in the whole boundary. That way we can build an open bounded convex $\C^{1,1}$ set $\Om$, such that the strictly convex part of the boundary in the sense of \eqref{eq:strict} is empty.
 
 This technical difficulty is due to the convexity constraint: it is difficult to exclude such a singular set $\Om$ from optimality, because it is hard to write optimality conditions around a set with a priori such poor regularity. Indeed, most of the perturbations of this shape becomes non-convex (and so are not admissible). Roughly speaking, this set saturates the convexity constraint almost everywhere; hypothesis \eqref{eq:H2} demands that the optimal shape do not saturate the convexity constraint on a nonempty part of the boundary.

Hypothesis \eqref{H2} is even stronger and requires that there is a finite number of alternation between saturated and non-saturated parts. Nevertheless, we do not know any proof that \eqref{H2} nor \eqref{eq:H2} is satisfied by an optimal shape for \eqref{eq:pb} (it is announced in \cite{HO} that there are only two segments in the boundary, but it seems that the proof is incomplete).
\end{remark}

\section{Proof of Theorem \ref{th:convex}}\label{sect:convex}


\noindent
{\bf Outline of the proof:} On strictly convex parts $\G$ of the boundary $\partial\Om^*$, we have the analytic equality $|\nabla u_2|_{|\G}=\Lambda$; on the complementary part of the boundary, we have segments, which is a geometric information. We want to prove that these two informations imply that the regularity of the junction between strictly convex parts and segments is either $\mathcal{C}^{1,\frac{1}{2}}$, or $\mathcal{C}^{2,\frac{1}{2}}$. To this end, we use the conformal parametrization $\phi$ of the set $\Om^*$ which has the same H\"older-regularity as the shape, and we prove that our analytical and geometrical informations give respectively on each side a regularity property on the Dirichlet and the Neumann boundary conditions, for the harmonic function $\log(|\phi'|)$ (the regularity of this function also characterizes the regularity of the shape); we then apply a result about mixed boundary problem (Lemma \ref{lem:dvp1}), which asserts that such a situation can only be satisfied when this function is either $\mathcal{C}^{0,\frac{1}{2}}$ or $\mathcal{C}^{1,\frac{1}{2}}$;
the shape is therefore $\mathcal{C}^{1,\frac{1}{2}}$ or $\mathcal{C}^{2,\frac{1}{2}}$. This last possibility is excluded by Proposition \ref{prop:nonregC2} below, taken from \cite{HO}.\\

\noindent\textbf{Proof of Theorem \ref{th:convex}:} A priori, we know that $\Om^*$ is necessarily of class $\C^1$ (see \cite{B}).\\

\noindent{\bf First step. Euler-Lagrange equation:}

Let $\Om^*$ be one solution of \eqref{eq:pb}. We can use Proposition \ref{prop:HO}, and so there exists one constant $\Lambda>0$ such that: 
\begin{equation}\label{EL}
 |\nabla u_2|_{|\G}=\Lambda,
\end{equation}
where $u_2$ is a normalized second eigenfunction and $\G\subset\partial\Om^*$ denotes the strictly convex parts of $\partial\Om^*$. We want to deduce from \eqref{EL} that $\Om^*$ is $\C^{1,\frac{1}{2}}$.

As we assume there is a finite number of segments, we can work locally around the intersection of a strictly convex part $\gamma_-\subset\G$, and a straight line $\gamma_+\subset\partial\Om^*\setminus\G$. 

So we focus on the following geometrical situation (see Figure \ref{conforme}):
\begin{itemize}
\item $\gamma_-\subset\partial\Om^*, \gamma_+\subset\partial\Om^*$, and $\overline{\gamma_-}\cap\overline{\gamma_+}$ is reduced to one point denoted by $A$,
\item $\gamma_-$ is strictly convex (in the geometrical sense \eqref{eq:strict}),
\item $\gamma_+$ is a segment.
\end{itemize}
We remind that \eqref{EL} implies that $u_{2}\in \C^1(\Om\cup\G)$ (since $\nabla u_{2}$ is in $H^1$ with a continuous trace on $\G$), and also that
$\G$ is very regular (see proposition \ref{prop:reg-over}), so we just need to analyze the regularity around $\overline{\G}\cap\overline{\partial\Om^*\setminus\G}$, composed by a finite number of points like $A$ here (strictly speaking, we do not need this result here, and actually this one could be a consequence of the proof given here, but we prefer to focus on the new part of the result, that is to say the regularity around $A$, and we put the emphasis of the regularity of $\G$ in section \ref{ssect:reg2} for the interested reader).
\\

\noindent{\bf Second step. Transport on a smooth domain:}

We introduce the conformal parametrization of $\Om^*$: the Riemann mapping theorem (see \cite{P} for example) asserts the existence of a biholomorphic function \linebreak$\phi : \H\rightarrow\Om^*$, where $\H=\{z\in\CC;\; Im(z)<0\}$ (where $Im$ denotes the imaginary part of the complex number $z$).
Moreover, from a result due to Caratheodory, we know that $\phi$ continuously extends to an homeophormism between the closures of $\H$ and $\Om^*$. Finally we can choose $\phi(0)=A$, the intersection point.

We set $J_+:=\phi^{-1}(\gamma_+)\subset\R=\partial\H$, $J_-:=\phi^{-1}(\gamma_-)$, and we can choose
$\mathcal{V}$ a bounded semi-neighborhood of 0 in $\H$ such that $\partial\mathcal{V}\cap\R\subset \overline{J_-}\cup\overline{J_+}$ (see Figure \ref{conforme}).
\begin{figure}[!ht]
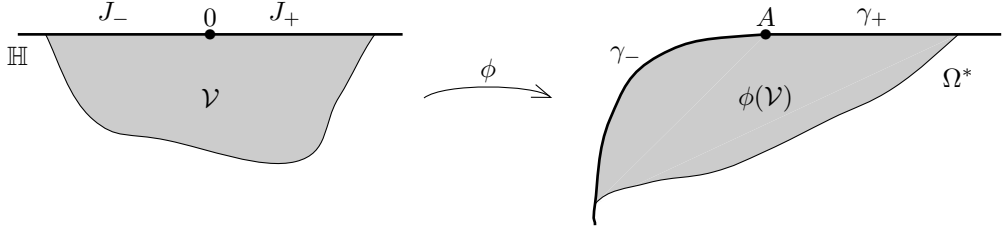


\figinit{0.8pt}
\figpt 0:(-215,20) 
\figpt 1:(-220,0)
\figpt 2:(-207,0)
\figpt 3:(-185,-40)
\figpt 4:(-150,-50)
\figpt 5:(-90,-60)
\figpt 6:(-70,-30)
\figpt 7:(-53,0)
\figpt 8:(-40,0)
\figpt 9:(-35,20) 

\figpt 10:(-130,0)
\figpt 11:(-175,0)
\figpt 12:(-95,0)
\figpt 13:(-130,-30)


\figpt 20:(130,0)
\figpt 21:(220,0)
\figpt 33:(240,0)
\figpt 22:(80,-10)
\figpt 23:(57,-40)
\figpt 24:(50,-80)
\figpt 34:(50,-90)
\figpt 25:(50,-100) 

\figpt 26:(45,-105) 
\figpt 27:(75,-70)
\figpt 28:(105,-65)
\figpt 29:(130,-55)
\figpt 30:(160,-40)
\figpt 31:(190,-25)
\figpt 32:(230,10) 

\figpt 35:(180,0)
\figpt 36:(130,-30)

\figpt 40:(-30,-30) 
\figpt 41:(-20,-23)
\figpt 42:(20,-23)
\figpt 43:(30,-30) 
\figpt 44:(0,-17) 

\psbeginfig{}
\psarrowBezier [40,41,42,43]

\psset(fillmode=yes,color=0.8)
\pscurve[0,2,3,4,5,6,7,9]
\psset(fillmode=no,color=0)
\psset(width=1)
\psline[1,8]
\psset(width=\defaultwidth)
\pscurve[0,2,3,4,5,6,7,9]

\psset(fillmode=yes,color=0.8)
\psline[20,21,24]
\pscurve[21,20,22,23,24,34]\pscurve[26,24,27,28,29,30,31,21,32]
\psset(fillmode=yes,color=1)
\pscurve[26,24,27,28]
\psset(fillmode=no,color=0)
\pscurve[26,24,27,28,29,30,31,21,32]
\psset(width=1)
\psline[20,33]
\pscurve[21,20,22,23,24,34,25]

\psendfig
\figvisu{\figBoxA}{}{
\figwriten 10:$0$(4)
\figwriten 11:$J_-$(4)
\figwriten 12:$J_+$(4)
\figwritec[13]{$\mathcal{V}$}
\figwrites 1:$\H$(6)
\figwriten 20:$A$(4)
\figwritew 22:$\gamma_-$(8)
\figwriten 35:$\gamma_+$(4)
\figwritec[36]{$\phi(\mathcal{V})$}
\figwritec[44]{$\phi$}
\figwrites 21:$\Om^*$(16)
\figsetmark{$\figBullet$}
\figwritep[10,20]
}
\centerline{\box\figBoxA}
\caption{Conformal parametrization}
\label{conforme}
\end{figure}

We can now transport the Dirichlet problem, initially settled in $\Om^*$, in the new domain $\H$ (which is smooth). Therefore, we put the ``unknown'' $\Om^*$ inside the equations: we set $\;\widehat{\cdot}\;$ the composition by $\phi$, and we get, since $\phi$ is holomorphic :
\begin{equation}\label{eq:over}
\left\{\begin{array}{ccll}
-\Delta \widehat{u_2} &=& \lambda_2(\Om^*) |\phi'|^2\widehat{u_2} & in\;\H,\\
\widehat{u_2}&=&0& on\;\partial\H,
\end{array}
\right.
\end{equation}\\
{\bf Third step. Regularity of $\phi$, from the regularity of $\Om^*$:}

We have to reformulate the question of the regularity of $\Om^*$ on $\phi$. We have the following lemma, available for the conformal parametrization of any convex $\C^1$ set:
\begin{lemma}\label{lem:conforme}If a open set $\Om$ is $\mathcal{C}^1$ (and simply connected), then its conformal parametrization $\phi$ satisfies:
\begin{itemize}
\item $\Arg(\phi')$ is defined and continuous on $\overline{\H}$,
\item for every $p\in [1,\infty)$, $|\phi'|\in L^p(\mathcal{V})$, for all $\mathcal{V}$ bounded subset of $\H$,
\end{itemize}
If a open set $\Om$ is convex, then its conformal parametrization $\phi$ satisfies:
\begin{itemize}
\item for all $\mathcal{V}$ bounded subset of $\H$, there exists $\beta>0$ such that $|\phi'|\geq\beta$ in $\mathcal{V}$.
\end{itemize}
\end{lemma}

We refer to \cite[Th 3.2]{P} and \cite[Ex 15 page 71]{GM05Har} for the first part. We refer to \cite[Ex 3.6.1 page 70]{P} for the second one (see also Lemma 1 in \cite{Ma} and references therein).
\begin{remark}
The regularity of $\Arg(\phi')$ is easy to understand since it is a parametrization of the angle of the tangent vector to the boundary of $\partial\Om^*$: indeed, $\phi_{|\R}:\R\to\partial\Om^*$ is a one-to-one parametrization of the boundary $\partial\Om^*$, and the tangent vector is given by $\frac{\phi'(t)}{|\phi'(t)|}=e^{i\Arg(\phi'(t))}$. Therefore $\Arg(\phi')$ is harmonic (as the imaginary part of the holomorphic function $\log(\phi')$) and has a continuous trace on the boundary and so is itself continuous up to the boundary.

Nevertheless, it is not true in general that this implies that $|\phi'|$ is continuous as well. The $L^p$ regularity of $|\phi'|$ is a consequence of results on conjugate functions: $\log(|\phi'|)$ and $Arg(\phi')$ are the real and imaginary parts of the same holomorphic function and so they are called conjugate, and it is well-known that their regularity are linked. Actually, it is possible to prove that $\phi'$ is in the Hardy space $H^p$, which is included in $L^p$.
\end{remark}

The situation is simpler about H\"older-regularity with non-integer exponents. 
\begin{lemma}[Kellog-Warschawski]\label{lem:conf}
Let $n\in\N$, $\beta\in(0,1)$, and $\gamma$ a relatively open subset of $\partial\Om^*$. Then
$$\log|\phi'|\textrm{ is }\C^{n,\beta}\textrm{ on }\phi^{-1}(\gamma) \Longleftrightarrow\Arg(\phi')\textrm{ is }\C^{n,\beta}\textrm{ on }\phi^{-1}(\gamma) \Longleftrightarrow\gamma\textrm{ is }\C^{n+1,\beta}.$$
\end{lemma}
See \cite{W,GM05Har,P}.\\

\noindent{\bf Fourth step. Regularity of $\Arg(\phi')$ on $J_+$, and regularity of $\log(|\phi'|)$ on $J_-$:}
 
On the one hand, we already noticed that $\Arg(\phi')$ is a parametrization of the angle of the tangent vector to the boundary of $\partial\Om^*$. Therefore, $\Arg(\phi')$ is a harmonic function with a constant (and so very regular) trace on $\overline{J_+}$ (since $\gamma_+$ is straight).

On the other hand, from the regularity results in Sobolev spaces for problem \eqref{eq:over} (we have an elliptic equation whose second member is in $L^p(\mathcal{V})$, using Lemma \ref{lem:conforme}), we get $\widehat{u_2}\in W^{2,p}(\mathcal{V})$ for every $p\in(1,\infty)$.
Therefore, since $|\nabla \widehat{u_2}|\in W^{1,p}(\mathcal{V})$ and $\widehat{|\nabla u_2|}=|\nabla u_2|\circ\phi$ is continuous on $\mathcal{V}\cup J_-$ by \eqref{EL},
the identity
\begin{equation}
|\nabla \widehat{u_2}| = |\phi'|\widehat{|\nabla u_2|},
\end{equation}
implies that $|\phi'|$ is defined and continuous on $\mathcal{V}\cup J_-$.\\
As a consequence, we can write
\begin{equation}\label{eq:EL^}
|\nabla \widehat{u_2}| = \Lambda|\phi'|,\;on\;J_-
\end{equation}
and this links the regularity of $|\phi'|$ on $J_-$ to the one of $\widehat{u_2}$. 
We need now to take the logarithm:
\begin{lemma}\label{lem:ln}
Let $v:\mathcal{U}\rightarrow\CC$ where $\mathcal{U}$ is an open bounded and Lipschitz domain of $\R^2$, and $p\in(1,\infty)$. Then
$$\left[\;v\in W^{1,p}(\mathcal{U})\textrm{ and }|v|\geq\beta>0\textrm{ in }\mathcal{U}\;\right]\Rightarrow \log(v)\in W^{1,p}(\mathcal{U}).$$
$$\left[\;v\in W^{2,p}(\mathcal{U})\textrm{ and }|v|\geq\beta>0\textrm{ in }\mathcal{U}\;\right]\Rightarrow \log(v)\in W^{2,p}(\mathcal{U}).$$
\end{lemma}
\begin{center}
\begin{minipage}{0.92\linewidth}
\textbf{Proof of Lemma \ref{lem:ln}:} Although this result is more or less classical, we give a short proof. We easily compute:
$$\partial_{x}(\log v)=\partial_{x}v\frac{1}{v},$$
with $\partial_{x}v\in L^p$ and $\frac{1}{v}\in L^\infty$, and using that $\mathcal{U}$ is bounded, we get that $\partial_{x}(\log v)\in L^p(\mathcal{U})$, and similarly for $\partial_{y}v$.

About the second order derivative, 
$$\partial^2_{xx}(\log v)=\partial^2_{xx}v\frac{1}{v}-(\partial_{x}v)^2\frac{1}{v^2}$$
with $\partial^2_{xx}v\in L^p(\mathcal{U}), \frac{1}{v^2}\in L^\infty(\mathcal{U})$, and finally $(\partial_{x}v)^2\in L^{\frac{p^*}{2}}(\mathcal{U})$ where $p^*=\frac{2p}{2-p}$ if $p<2$, and $p^*=\infty$ if $p> 2$, and if p=2, $(\partial_{x}v)^2\in L^q$ for every $q<\infty$ (Sobolev imbeddings).
In all cases, $(\partial_{x}v)^2\in L^p$, and so we get $\partial^2_{xx}(\log v)\in L^p(\mathcal{U})$.
The case of $\partial^2_{yy}(\log v)$ is exactly the same.

About the case of $\partial^2_{xy}(\log v)$, using the same strategy, it remains to prove that the property $\partial_{x}v, \partial_{y} v\in L^{p^*}$ implies that  the product $\partial_{x}v\partial_{y} v\in L^{p}$, in the case $p<2$ (the cases $p=2$ and $p>2$ are easy). To that hand, we apply H\"older inequality for $q=\frac{p^*}{p}\in(1,+\infty)$ whose conjugate exponent is $q'=\frac{2}{p}$:
$$\int (\partial_{x}v \partial_{y} v)^p \leq \left(\int (\partial_{x}v)^{pq}\right)^{\frac{1}{q}} \left(\int (\partial_{y}v)^{pq'}\right)^{\frac{1}{q'}}<\infty$$
these terms being finite since $pq=p^*$ and $pq'=2\leq p^*$.
\qed
\end{minipage}\end{center}
\vspace{-0.1cm}Thus, using that $|\phi'|$ do not vanish on $\overline{\mathcal{V}}$ (see Lemma \ref{lem:conforme}), combined with \eqref{eq:EL^}, the continuity of $\nabla \widehat{u_{2}}$, and the first part of this lemma, we get that $\log(\frac{|\nabla \widehat{u_2}|}{\Lambda})\in W^{1,p}(\mathcal{V})$ (if $\mathcal{V}$ is small enough such that $\nabla \widehat{u_2}$ do not vanish on $\mathcal{V}$) and then with \eqref{eq:EL^}, $\log(|\phi'|)_{|J_-}\in W^{1-\frac{1}{p},p}(J_-\cap\partial\mathcal{V}), \forall p\in [1,+\infty).$\\

\noindent{\bf Fifth step. Regularity for a mixed problem:}

Setting $a:=\log|\phi'|=Re(\log(\phi'))$ and $b:=\Arg(\phi')=Im(\log(\phi'))$ (linked by Cauchy-Riemann equations which can be extended to $J_{+}$  by regularity),
 we now deal with the following problem:
\begin{equation}\label{eq:mixte}
\left\{\begin{array}{ccll}
\Delta a &=&0 & in\;\H\\
a&=&\log\left(\frac{|\nabla \widehat{u_{2}}|}{\Lambda}\right)=:g_D& on\;J_-\\
\partial_y a &=&- \partial_x b=:g_N& on\;J_+.
\end{array}
\right.
\end{equation}
We use the following lemma, dealing with the asymptotic expansion of the solution of a mixed boundary value problem on a domain with a flat corner:
\begin{lemma}\label{lem:dvp1}
Let $a\in L^p(\mathcal{V})$ satisfy (\ref{eq:mixte}), $p>4$, and $\mathcal{V}'\subset\mathcal{V}$ such that $\overline{\mathcal{V}'}\subset\mathcal{V}\cup\partial\H$.

If $g_D\in W^{1-\frac{1}{p},p}(J_-)$ and $g_N\in \mathcal{C}^\infty(\overline{J_+})$, then $\exists\;a_0\in\R$ such that
\begin{equation}\label{eq:dvp1}
a-a_0 r^{\frac{1}{2}}\cos(\frac{\varphi}{2})
\in W^{1,p}(\mathcal{V}')
\end{equation}
where $(r,\varphi)$ are polar coordinates, centered at $0$, and such that $\varphi=0\;on\;J_+$ and $\varphi=\pi\;on\;J_-$.

If moreover, $g_D\in W^{2-\frac{1}{p},p}(J_-)$, then $\exists\;a_1\in\R$ such that
\begin{equation}\label{eq:dvp2}
a-a_0 r^{\frac{1}{2}}\cos(\frac{\varphi}{2})-a_1r^{\frac{3}{2}}\cos(\frac{3\varphi}{2})
\in W^{2,p}(\mathcal{V}').
\end{equation}
\end{lemma}
(see \cite[Th 5.1.3.5]{G}, and Remark \ref{rem:mixte} below for more comments and references).

We can apply the first part of this lemma to $a=\log|\phi'|$, using the previous steps.

First, we get the asymptotic expansion \eqref{eq:dvp1} and thus $\log|\phi'|\in\C^{0,\frac{1}{2}}(\overline{\mathcal{V}'})$, which is the regularity of the function $(r,\varphi)\mapsto r^{\frac{1}{2}}\cos(\frac{\varphi}{2})$ (because for $p$ large enough, $W^{1,p}$ is included in $\C^{0,\frac{1}{2}}$). Applying this result around any point of $\overline{\G}\cap\overline{\partial\Om^*\setminus\G}$, we can get that $\log(|\phi'|)$ is $\C^{0,\frac{1}{2}}$ on the whole $\overline{\H}$, and using Lemma \ref{lem:conf}, we get the announced statement about the regularity of $\Om^*$, namely that this one is  $\C^{1,\frac{1}{2}}$.\\

\noindent{\bf Sixth step. Non-regularity result:}

Now, if we assume by contradiction that $\Om^*$ is $\C^{1,\frac{1}{2}+\eps}$, then we get that $\log|\phi'|$ is $\C^{0,\frac{1}{2}+\eps}$, using again Lemma \ref{lem:conf}, and thus necessarily $a_0=0$ in the asymptotic expansion \eqref{eq:dvp1}.

Therefore $\log|\phi'|\in W^{1,p}(\mathcal{V}')$, but since $|\phi'|\in L^\infty$, this easily implies that $|\phi'|\in  W^{1,p}(\mathcal{V}')$. We can apply again the previous steps 4 and 5: since $|\phi'|\in W^{1,p}(\mathcal{V}')$, \eqref{eq:over} gives $\widehat{u_{2}}\in W^{3,p}(\mathcal{V}')$ and \eqref{eq:EL^} gives, using the second part of Lemma \ref{lem:ln}, $\log|\phi'|\in W^{2-\frac{1}{p},p}(J\cap\partial\mathcal{V}')$; we use now the expansion \eqref{eq:dvp2} in Lemma \ref{lem:dvp1}, with $a_0=0$, to get that $\log|\phi'|$ is $\C^{1,\frac{1}{2}}$ around $0$ (in a neighborhood $\mathcal{V}''$ such that $\overline{\mathcal{V}''}\subset \mathcal{V}'\cup\partial\H$). As in the previous step, we finally get that $\log(|\phi'|)$ is $\C^{1,\frac{1}{2}}$ on $\overline{\H}$, and thus $\Om^*$ is $\C^{2,\frac{1}{2}}$. This last property is a contradiction with the following non-regularity result, proved in \cite[Th 10]{HO}.\qed
\begin{prop}[Henrot-Oudet]\label{prop:nonregC2}
Let $\Om^*$ be a solution of \eqref{eq:pb}. Then, for every $\eps>0$, $\Om^*$ is not $\C^{2,\eps}$.
\end{prop}
\begin{remark}\label{rem:nonregC2}
The idea for proving this proposition(\cite[Th 10]{HO}) is to count the number of nodal domains of $\partial_{x}u_{2}$ where the direction $x$ is chosen as the direction of one segment of $\partial\Om^*$ touching the nodal line of $u_{2}$. There is a small gap in the proof given by Henrot and Oudet, since they use that the boundary of $\Om^*$ only contains two segments, and that these ones are parallel; but even if these geometrical properties are announced in \cite{HO}, it remains a gap in the proof that there are only two segments (part ``At most two segments'' in \cite[Th 9]{HO}), and so these properties are still open. However, their proof of Proposition \ref{prop:nonregC2} can be easily adapted with minor revisions to our context (it suffices to count the number of nodal domains of $\partial_{x}u_{2}$ in more general cases).

Note that with more work and arguments of a completely different nature,
we can improve Proposition  \ref{prop:nonregC2} and Theorem \ref{th:convex} and prove that $\Omega^*$ is
actually not $\C^{1,\frac{1}{2}+\epsilon}$ without assumption \eqref{H2} (this will be done
in \cite{L0}). However, so far, we do need \eqref{H2} to prove the $\C^{1,\frac{1}{2}}$
regularity.
\end{remark}

\section{Remarks and extensions}

We give here a few comments on the proof given in the previous section, especially we briefly describe the general framework of overdetermined problem wherein our result can be generalized. We also deduce an application to the regularity of some optimal shapes, in the same spirit of Theorem \ref{th:convex}.

The proof given in Section \ref{sect:convex} only uses the Euler-Lagrange equation (\ref{eq:EL}), together with the fact that $u_{2}$ is an eigenfunction, that is to say:
\begin{equation}
\left\{\begin{array}{ccll}
-\Delta u_2 &=& \lambda_2(\Om^*) u_2 & in\;\Omega^*\\
u_2&=&0& on\;\partial\Omega^*\\
|\nabla u_2| &=& C^{st}=\Lambda>0& on\; \Gamma,
\end{array}
\right.
\end{equation}
where $\Gamma$ is a relatively open subset of $\partial\Om^*$. 
This kind of system is called a partially overdetermined problem; the third equation is the overdetermined part, and is supposed to give some information about the domain $\Om^*$. We refer to the paper \cite{FG} for some symmetry results about this kind of problems (the word ``partially'' means that the overdetermined equation is only valid on a part of the boundary). We focus here on the question of regularity, more precisely the regularity around $\overline{\Gamma}\cap\overline{\partial\Om\setminus\partial\Gamma}$ (see subsection \ref{ssect:reg1}); our method can easily be iterated, see Proposition \ref{prop:reg_boot}. It was already known that the overdetermined equation implies that $\Gamma$ is regular, but we show in subsection \ref{ssect:reg2} that the strategy introduced in this paper can produce a new proof of that result (in dimension 2 only). We apply these results to some other shape optimization problems, and we conclude the paper with some remarks and perspectives.

\subsection{Regularity of partially overdetermined problems}\label{ssect:reg1}

With the help of a similar analysis of tools used in the proof given in Section \ref{sect:convex}, we can get the following result, which deals with the regularity around the intersection of the overdetermined part and the remaining boundary:
\begin{prop}\label{prop:reg}
Let $\Omega$ be an open bounded set of $\R^2$, and $\Gamma$ a relatively open subset of $\partial\Om$.
We assume that
\begin{itemize}\item $\G$ has a finite number of connected components,
\item $\partial\Om$ is $\mathcal{C}^{1}$, and $\overline{\partial\Om\setminus\G}$ is $\mathcal{C}^\infty$.
\end{itemize}
Finally we assume there exists $u\in \C^2(\Omega)\cap\C^1(\Omega\cap\G)\cap L^\infty(\Om)$ satisfying
\begin{equation}\label{eq:over2}
\left\{\begin{array}{ccll}
-\Delta u &=& f(u) & in\;\Omega\\
u&=&0& on\;\partial\Omega\\
|\nabla u| &=& C^{st}=\Lambda>0& on\; \Gamma,
\end{array}
\right.
\end{equation}
where $f:\R\to\R$ is a $\C^{\infty}$ function, and assuming also $f(u)\geq0$ in a neighborhood of $\overline{\Gamma}\cap\overline{\partial\Om\setminus\G}$.
Then,
\begin{itemize}
\item either $\partial\Omega$ is $\mathcal{C}^{1,\frac{1}{2}}$ and $\forall\;\eps>0,\;\partial\Omega$ is not $\mathcal{C}^{1,\frac{1}{2}+\eps}$,
\item or $\partial\Omega$ is $\mathcal{C}^{2,\frac{1}{2}}.$
\end{itemize}
\end{prop}
\begin{remark}
The same result is true if we replace the assumption $f(u)\geq 0$ by $\Om$ is convex, or also by $\Om$ is $\C^{1,\alpha}$ for some $\alpha\in(0,1)$. Indeed, the fact that $f(u)$ is positive is used to get that $\nabla \widehat{u}$ cannot vanish on $\partial\H$. When $\Om$ is convex or $\C^{1,\alpha}$, its conformal parametrization has a derivative which cannot vanish, which leads to the same conclusion (as in the proof of Theorem \ref{th:convex}).
\end{remark}

\noindent{\bf Sketch of proof: }the strategy is exactly the same as in section \ref{sect:convex}: 
again, since $\G$ is regular (see Proposition \ref{prop:reg-over}), we work around a point $A\in \overline{\Gamma}\cap\overline{\partial\Om\setminus\G}$. We choose $\phi$ a conformal parametrization of $\Om$ such that $\phi(0)=A$, and we consider, as in the previous section, $\gamma_-\subset\Gamma$,
$\gamma_+\subset\partial\Om\setminus\Gamma$ connected and such that $\overline{\gamma_-}\cap\overline{\gamma_+}=A$, and we denote $J_{\pm}:=\phi^{-1}(\gamma_{\pm})$. We also consider a bounded semi-neighborhood $\mathcal{V}$ of 0 in $\overline{\H}$.\\
Since $\Om$ is $\C^{1}$, its conformal parametrization $\phi$ is such that $|\phi'|\in L^p(\mathcal{V})$ for all $p\in(1,\infty)$ (see Lemma \ref{lem:conforme}). Moreover, $\widehat{u}:=u\circ\phi$ is solution of
\begin{equation}\label{eq:over3}
\left\{\begin{array}{ccll}
-\Delta \widehat{u} &=& |\phi'|^2f(\widehat{u}) & in\;\H,\\
\widehat{u}&=&0& on\;\partial\H.
\end{array}
\right.
\end{equation}\\
Since $u$ is bounded and $f$ is continuous, $f(u)$ is also bounded, and so $\widehat{u}$ is in $W^{2,p}(\mathcal{V})$ for all $p\in(1,\infty)$ (and so the gradient of $\widehat{u}$ is continuous on $\overline{\mathcal{V}}$). Using that $f(u)\geq 0$ and strong maximum principle, we get that $\nabla \widehat{u}$ cannot vanish on $\partial\H\cap\overline{\mathcal{V}}$, and so on $\overline{\mathcal{V}}$ by continuity (we might need to reduce the neighborhood $\mathcal{V}$ here). Using Lemma \ref{lem:ln}, $\log\frac{|\nabla \widehat{u}|}{\Lambda}\in W^{1,p}(\mathcal{V})$.

Therefore, $a:=\log(|\phi'|)$ satisfies \eqref{eq:mixte} with $g_{D}=\log(\frac{|\nabla\widehat{u}|}{\Lambda})\in W^{1-\frac{1}{p},p}(J_-)$ and $g_{N}=-\partial_{x}\arg(\phi')\in\C^\infty(\overline{J_{+}})$. With the regularity Lemma \ref{lem:dvp1} on mixed problems, we get $a\in\C^{0,\frac{1}{2}}$ on a neighborhood of 0 in $\overline{\H}$, and since $A$ is any point of $\overline{\Gamma}\cap\overline{\partial\Om\setminus\G}$, $\partial\Om$ is globally $\C^{1,\frac{1}{2}}$.\\
If now we assume that $\Om$ is $\C^{1,\frac{1}{2}+\eps}$, this means that the first term (the one in $r^\frac{1}{2}$) is the asymptotic development of $a$ is $0$, and repeating the same arguments as before, with one more rank in the regularity, we finally get that $\partial\Om$ is $\C^{2,\frac{1}{2}}$.\qed

\begin{remark}
It is sufficient that $f$ be $C_{loc}^{0,1}$; the same proof is valid, we just have to be careful on the regularity of $\G$ and so the one of $\Arg(\phi')$ on $J_{+}$ and use a generalized version of Lemma \ref{lem:dvp1}.
\end{remark}

\subsection{Bootstrap of the strategy}

It is easy to see that the strategy used in section \ref{sect:convex} can be iterated to get the following generalization:
\begin{prop}\label{prop:reg_boot}
Under the same assumptions as in Proposition \ref{prop:reg}, we have
\begin{itemize}
\item either $\partial\Omega$ is $\mathcal{C}^\infty$,
\item or $\exists\;k\in\mathbb{N}^*,\;\left[\;\partial\Omega\in \mathcal{C}^{k,\frac{1}{2}}
,\textrm{ and }\forall\eps>0, \partial\Omega\notin\mathcal{C}^{k,\frac{1}{2}+\eps}\;\right].$
\end{itemize}
\end{prop}
See \cite{L} for a proof, whose strategy is to iterate the main steps as in the proofs of Theorem \ref{th:convex} and Proposition \ref{prop:reg}, and using the following lemma, very similar to Lemma \ref{lem:dvp1} but adapted to H\"older spaces rather than Sobolev ones, and with more rank in the development.
\begin{lemma}\label{lem:dvp2}
Let $a\in \C^{0,\beta}(\mathcal{V})$ be solution of \eqref{eq:mixte}, $\beta\in(0,1)\setminus \{\frac{1}{2}\}$, and $\mathcal{V}'\subset\mathcal{V}$ such that $\overline{\mathcal{V}'}\subset\mathcal{V}\cup\partial\H$.\\
If $g_D\in \C^{n,\beta}(\overline{J_-})$ with $n\in\N^*$, and $g_N\in \mathcal{C}^\infty(\overline{J_+})$, then $\exists\;a_0\ldots,a_n\in\R$ such that
\begin{eqnarray}\label{eq:dvp3}
\textrm{if }\beta<\frac{1}{2},&\displaystyle{a-\sum_{i=0}^{n-1} a_i r^{i+\frac{1}{2}}cos\left((i+\frac{1}{2})\varphi\right)}& \in \C^{n,\beta}(\overline{\mathcal{V}'})\\
\textrm{if }\beta>\frac{1}{2},&\displaystyle{a-\sum_{i=0}^{n} a_i r^{i+\frac{1}{2}}cos\left((i+\frac{1}{2})\varphi\right)}& \in \C^{n,\beta}(\overline{\mathcal{V}'}).
\end{eqnarray}
where $(r,\varphi)$ are polar coordinates, centered at $0$, and such that $\varphi=0\;on\;J_+$ and $\varphi=\pi\;on\;J_-$.
\end{lemma}
(see \cite[Ths 6.4.2.6]{G})

\subsection{Remark on the regularity of the overdetermined part}\label{ssect:reg2}

As we said in the beginning of the proof of Theorem \ref{th:convex}, our strategy using the conformal parametrization to analyze the regularity on the extremities of an overdetermined part, can also be used to get the regularity inside the overdetermined part (in dimension two). More precisely, we give a short proof of the following result, see \cite{V} for a more general statement.

\begin{prop}\label{prop:reg-over}
Let $\Om\subset\R^2$, $\G\subset\partial\Om$ relatively open and of class $\C^{1}$, and $f\in\C^\infty(\R)$, such that there exists $u\in\C^1(\Om\cup\G)\cap\C^2(\Om)$ solution of  \eqref{eq:over2}. Then $\G$ is of class $\C^{\infty}$, and $u\in\C^\infty(\Om\cup\G)$.
\end{prop}
\begin{remark}\label{rem:reg-over-k}
If we only assume $f\in \C^{k,\alpha}(\R)$ for some $k\in\N$ and $\alpha\in(0,1)$, we get $\G$ of class  $\C^{k+2,\alpha}$.
\end{remark}

\begin{proof}
As stated  in the previous subsection, our method enlightens a bootstrap in the regularity, given by the equation $|\nabla u|=C^{st}$: at the first step, $\Gamma$ is $\C^{1}$ and so, if $\phi$ is a conformal parametrization, $|\phi'|\in L^p(\mathcal{V})$ where $\mathcal{V}$ is a semi neighborhood of any point of $\phi^{-1}(\Gamma)$; this implies that $\widehat{u}$, solution of \eqref{eq:over3}, is $W^{2,p}(\mathcal{V})$. Then the overdetermined equation 
\begin{equation}
|\nabla \widehat{u_2}| = |\phi'|\widehat{|\nabla u_2|}=\Lambda|\phi'|\; on\; \phi^{-1}(\Gamma),
\end{equation}
gives $|\phi'|\in W^{1,p}(\mathcal{V})$, and so is $\log|\phi'|$ (since $|\nabla u|>0$ in $\G$ implies that $f(u)$ cannot change its sign around $\G$, and neither do $f(\widehat{u})$ and $|\nabla\widehat{u}|$) and thus
$\Gamma$ is $\C^{1,\alpha}$ and $u\in\C^{1,\alpha}(\Om\cup\G)$ for all $\alpha<1$. Using again the same strategy, we then get $\G\in C^{2,\alpha}$ and $u\in\C^{2,\alpha}(\Om\cup\G)$. This technique can easily be iterated, and gives the $\C^\infty$ regularity.\qed
\end{proof}
\begin{remark}\label{rem:reg-over}
We can find a similar regularity result in a more general setting (and non necessarily 2-dimensional) in \cite{V} (see also \cite[Th. 16]{FG}). Using moreover \cite{KN}, we get in fact that $\Gamma$ is analytic when $f$ is analytic (this is the case for the optimal shape for problem \eqref{eq:pb} where $f$ is linear).
\end{remark}

\subsection{Application to the regularity of optimal shapes}

The link between shape optimization and overdetermined problem is clear and has already been used in this paper: the overdetermined equation $|\nabla u|_{|\Gamma}=C^{te}>0$ can often be seen as the first optimality condition for the optimization of classical shape functionals under volume constraint, where $u=u_{\Om^*}$ is the state function of an optimal shape, and $\Gamma\subset\partial\Om^*$ is the part of the boundary which do not saturates the other constraints of the problem (if there is no other constraint in the optimization, then $\Gamma=\partial\Om^*$ and we have a classical overdetermined problem; there are some symmetry results like Serrin's one asserting that $\Om^*$ is necessarily a ball in such situation, see \cite{FG} for references on that topic). Problem \eqref{eq:pb} is an example of this situation: we get the overdetermined equation on the strictly convex parts of the boundary.

As new examples of this situation, we now analyze the regularity question on the following shape optimization problems:

\begin{minipage}{7 cm}
 \begin{equation}\label{eq:minla1}
 \min_{\stackrel{|\Om|=V_0}{\Om\subset D}} \la_1(\Om)
\end{equation}
\end{minipage}
\begin{minipage}{7 cm}
 \begin{equation}\label{eq:minJ}
 \min_{\stackrel{|\Om|=V_0}{\Om\subset D}} J(\Om)
\end{equation}
\end{minipage}\\[3mm]
where 
$J(\Om)$ is the Dirichlet energy of $\Om$ for the right hand side $1$; this means that $J(\Om)=\int_\Om \frac{1}{2} |\nabla u_\Om|^2-u_\Om$ where $u_{\Om}$ is the unique variational solution of
\begin{equation}\label{eq:lap}
u_\Om\in H^1_0(\Om),\;\;-\Delta u_\Om =1\;in\;\Om.
\end{equation}
 The functional $J(\Om)$ can also be defined by: 
 \begin{equation}\label{eq:energie}
  J(\Om)=\min_{v\in H^1_0(\Om)}\left\{\int_\Om \frac{1}{2} |\nabla v|^2-v\right\}.  
 \end{equation}
Here,
 $D$ is a bounded open set (a box), and we deduce from Proposition \ref{prop:reg_boot} the following result:
 
 \begin{prop}\label{prop:opti}
Let $V_0>0$ and $D$ a $\C^\infty$ open subset of $\R^2$. Let $\Omega^*\subset\R^2$ be a solution of \eqref{eq:minla1} or \eqref{eq:minJ}.\\
We assume:
\begin{itemize}
\item $\partial\Om^*\cap\partial D$ has a finite number of connected components,
\item any contact between $\partial\Om^*$ and $\partial D$ is tangential.
\end{itemize}
Then
\begin{itemize}
\item either $\partial\Omega^*$ is $\mathcal{C}^\infty$,
\item or $\exists\;k\in\mathbb{N}^*,\;\left[\;\partial\Omega^*\in \mathcal{C}^{k,\frac{1}{2}}
,\textrm{ and }\forall\eps>0, \partial\Omega^*\notin\mathcal{C}^{k,\frac{1}{2}+\eps}\;\right].$
\end{itemize}
\end{prop}
See Remark \ref{rem:reg} for a discussion about the regularity assumption of the contact between the optimal shape and the box $D$.\\

\begin{proof}
We define the free boundary $\Gamma:=\partial\Om^*\cap D$; this one is very regular as proven in \cite{B04Reg,BL}, and one can write the optimality condition $|\nabla u|_{\Gamma}=\Lambda$, where $\Lambda>0$ is a Lagrange multiplier for the volume constraint, and $u$ is either the first eigenvalue of the Dirichlet-Laplacian if we consider \eqref{eq:minla1}, or the solution of \eqref{eq:lap} if we consider \eqref{eq:minJ}. So $u$ satisfies a partially overdetermined problem like \eqref{eq:over2} with $f(u)=\lambda_1(\Om^*)u$ or $f(u)=1$. In both cases, $f(u)\geq0$, and thus $\Om^*$ satisfies assumptions of Proposition \ref{prop:reg}. Therefore $\Om^*$ is $\C^{k,\frac{1}{2}}$ with $k\in \mathbb{N}^*\cup\{\infty\}$.\qed
\end{proof}

We now focus on a particular case for the box $D$, where one can identify the exponent $k$ appearing in Proposition \ref{prop:opti}.
\begin{prop}
Let $V_0>0$ and $D=\R\times(-M,M)$ for some $M>0$. Let $\Omega^*\subset\R^2$ be a solution of \eqref{eq:minla1} or \eqref{eq:minJ}.\\
We assume that the contact between $\partial\Om^*$ and $\partial D$ is tangential. Then
\begin{itemize}
\item either $\Om^*$ is a disk,
\item or $\left[\;\partial\Omega^*\in \mathcal{C}^{1,\frac{1}{2}}
,\textrm{ and }\forall\eps>0, \partial\Omega^*\notin\mathcal{C}^{1,\frac{1}{2}+\eps}\;\right].$
\end{itemize}
\end{prop}
\begin{proof}

 It is well known that the solution of \eqref{eq:minla1} or the one of \eqref{eq:minJ} is the ball of volume $V_0$, if this one is admissible (included in $D$). If such a ball does not exist, one can prove that any optimal shape $\Om^*$ should touch the boundary of the box (see \cite[Th 3.4.1]{H}: this is an easy consequence of Serrin's symmetry result if one knows the regularity of the free boundary, proven in \cite{B04Reg,BL}). 
Since the cylindrical box $D=\R\times(-M,M)$ has two orthogonal symmetry axes, one can prove using two Steiner symmetrization that $\Om^*$ also has two axes of symmetry, and therefore the free boundary necessarily has two connected components (see \cite{FGLP}), and the remaining boundary $\partial\Om^*\cap\partial D$ is the union of two segments. 
Thus, applying Proposition \ref{prop:reg}, we get  that $\partial\Omega^*$ is $\mathcal{C}^{1,\frac{1}{2}}$ or $\mathcal{C}^{2,\frac{1}{2}}$.\\
We exclude this last case with a proposition similar to Proposition \ref{prop:nonregC2} for problems \eqref{eq:minla1} and \eqref{eq:minJ}, see \cite{L,L0}\qed\end{proof}

\begin{remark}
We finally notice that this kind of regularity/singularity can be observed numerically as it is shown in \cite{L}.
\end{remark}

\subsection{Concluding remarks and perspectives}

\begin{remark}
In our mind, the non-regularity result (\ref{eq:result}) is surprising.
For instance, remind that if we consider the classical isoperimetric problem
$$P(\Omega^*)=\min_{\stackrel{|\Omega|=V_0} {\Omega\subset D}} P(\Omega),$$
where $D$ is regular enough and $P$ denotes the perimeter, the $\mathcal{C}^{1,1}$-regularity holds, as proved in $\cite{SZ}$.
In dimension 2, this result is easier, since the boundary of the optimal set is only made of pieces of
$\partial D$ and of arcs of circle with tangential contacts (free boundaries are regular and have a constant mean curvature).
\end{remark}
\begin{remark}\label{rem:reg}
{\bf On assumptions in Proposition \eqref{prop:opti}:} It seems not easy to prove this property with our strategy  based on the conformal parametrization. However, this property is certainly true, and there exist some results of that kind about the obstacle problem which could give a way to prove this property (see e.g. \cite{SU}).
\end{remark}

\begin{remark}\label{rem:mixte}
Lemma \ref{lem:dvp1} gives the asymptotic expansion for solutions of mixed elliptic problems, in a regular domain. Actually, this is a particular case of results dealing with asymptotic expansion of solutions to elliptic PDE, with Dirichlet and/or Neumann conditions, on domains with corners (here, the ``corner'' is flat, the corresponding angle is $\pi$). There is a profuse literature on that question, see for example the books \cite{G} and \cite{D88Ell}.

Statements are technical, but the idea is rather simple: for a mixed problem, we know there exist some non-regular solutions, even with smooth boundary conditions, namely
\begin{equation}\label{eq:solsing}
r^{n+\frac{1}{2}}cos\left(\left(n+\frac{1}{2}\right)\varphi\right),
\end{equation}
where $(r,\varphi)$ are polar coordinates around the meeting point of the Dirichlet condition and the Neumann one, chosen such that $\varphi=0$ on the side of Neumann condition, and $\varphi=\pi$ on the side of the Dirichlet one, and $n\in\mathbb{Z}$ (negative values of $n$ are excluded if we only consider solutions in $H^ 1$).

But above all, we know that any solution admits an asymptotic expansion around this junction point of Dirichlet and Neumann conditions, this expansion being a linear combination of these non regular solutions. Therefore, we get an asymptotic development up to a certain order of any solution of \eqref{eq:mixte}, this order being determined by the maximal regularity we can expect with the boundary conditions, and this regularity will be the one of the rest in the asymptotic expansion.

In particular, Lemma \ref{lem:dvp1} is announced for a function in  $L^p$ with $p>4$, and not in $H^1$. This requires to be careful, since the variational formulation is usually settled in $H^1$. Nevertheless, the default of uniqueness {\it below} $H^1$ is known for these problems: we know that every solution is a linear combination of the non regular solutions \eqref{eq:solsing}. The ones whose index $n$ is negative are excluded if we consider a solution in $L^p$ with $p>4$ (because $r^{-\frac{1}{2}}$ is not in $L^p$ if $p>4$), and therefore we get uniqueness in that spaces (see for example \cite{S68Reg}).
\end{remark}

\noindent{\bf Perspective}:
The final gap which has to be overcome about \eqref{eq:pb} is to prove that $\partial\Omega^*$ has a finite number of segments in its boundary, or possibly to treat the case of an infinite number of segments. In \cite{L0}, we choose this second way, and we extend the proof of this paper to get the negative part of Theorem \ref{th:convex} without any assumption, namely that $\Om^*$ is not $\C^{1,\frac{1}{2}+\eps}$.\\
Nevertheless, it seems natural to expect that $\Om^*$ has two orthogonal symmetry axes, and contains only two segments in its boundary, but these properties are still open.

\bibliographystyle{plain}

\end{document}